\documentclass[11pt]{article}

\usepackage{latexsym}
\usepackage{amsfonts, amsmath}
\usepackage{amssymb}
\usepackage{theorem}
\usepackage{longtable}
\usepackage{soul}
\usepackage[usenames]{color}
\usepackage{hyperref}
\usepackage{times}
\usepackage{tikz}
\usepackage{pgfplots}
\pgfplotsset{compat=1.15} 
\usetikzlibrary{patterns}
\usepackage{enumitem}
\usepackage{bbm}
\usepackage{multicol}
\usepackage{float}

\usepackage[margin=2cm]{geometry}

\newtheorem{theorem}{Theorem}[section]
\newtheorem{definition}[theorem]{Definition}
\newtheorem{remark}[theorem]{Remark}
\newtheorem{proposition}[theorem]{Proposition}
\newtheorem{corollary}[theorem]{Corollary}

\newtheorem{example}[theorem]{Example}

\def\ND{\succsim\kern-11pt/\kern5pt}

\newenvironment{proof}
{\begin{trivlist}\item[]{Proof:}}{\hfill{$\square$}\noindent\end{trivlist}}

\usepackage{epsfig}
\usepackage{tikz}
\usetikzlibrary{matrix,arrows,calc,patterns}
\tikzset{every scope/.style={>=angle 60,thick}}
\usepackage{etex}
\usepackage{tikz}
\usepackage{pgfplots}
\usepackage{verbatim}
\usepackage{graphicx} 
\usepackage{subfigure} 
\usepackage{float}

\newcommand{\Conv}{%
  \mathop{\scalebox{3}{\raisebox{-0.2ex}{$*$}}
  }
}

\newcommand{\SG}{\operatorname{SG}}
\newcommand{\SGs}{\widehat{\operatorname{SG}}}
\newcommand{\N}{\mathbb{N}}

\newcommand{\R}{\mathbb{R}}

\title{Simple games with minimum}
\author{Sascha Kurz$^1$ and Dani Samaniego$^2$\\ \footnotesize $^1$University of Bayreuth, \footnotesize sascha.kurz@uni-bayreuth.de
\vspace{-0.2cm}
\\  \footnotesize $^2$Universitat Politècnica de Catalunya, 
\footnotesize daniel.samaniego.vidal@upc.edu}
\date{ }

\begin{document}
\maketitle

\begin{abstract}
Every simple game is a monotone Boolean function. For the other direction we just have to exclude the two constant functions. The enumeration of monotone Boolean functions with 
distinguishable variables is also known as the Dedekind's problem. The corresponding number for nine variables was determined just recently by two disjoint research groups (\cite{JA23,Hi24}). Considering permutations of the variables as symmetries we can also speak about non-equivalent monotone Boolean functions (or simple games). For the recent determination of the number of non-equivalent monotone Boolean functions with nine variables and a summary of the previous history on the problem, we refer to \cite{paw24}. Here we consider simple games with minimum, i.e., simple games with a unique minimal winning vector. A closed formula for the number of such games is found as well as its dimension in terms of the number of players and equivalence classes of players. 
\end{abstract}

\section{Introduction}

Dedekind's problem is an open problem since 1897 (\cite{de1897}). This problem consists on finding a closed formula, in terms of the number of players, for the monotone Boolean functions with distinguishable variables. Several partial results have been obtained since then. On the one hand, the more computational approach searches the biggest known number of the sequence at hand. For example, the $9$th number of the Dedekind sequence, that is, without removing equivalent monotone functions, was determined recently and independently by Jäkel \cite{JA23} and Van Hirtum et. al. \cite{Hi24}. This term was added to the known terms of the sequence in \href{https://oeis.org/A000372}{A000372} at the On-Line Encyclopedia of Integer Sequences (OEIS). For the recent determination of the number of non-equivalent monotone Boolean functions with nine variables and a summary of the previous history on the problem, we refer to \cite{paw24}. Analogously, this late term is currently the last known of the sequence \href{https://oeis.org/A000372}{A003182} that contains the known non-equivalent monotone Boolean functions.

On the other hand, closed formulas have been found for some subsets of monotone Boolean functions. This theoretical approach has the advantage of helping to understand the structure of monotone Boolean functions and simple games. This paper focuses on this second approach (see \cite{CaFr96,FrKu13,FrMoRo2012,FrSa21DAM,KuSa23DAM,KuTau2013}  for other examples of this approach related to simple games).

Concretely, in this paper, we focus on simple games with a unique minimal winning vector. When focusing on this subset of the simple games (or monotone Boolean functions), we are able to find a closed formula in terms of the number of players for the non-isomorphic games of this kind. To achieve this goal, we translate the structure of the simple games parameterizing them in terms of a matrix and a vector. We want to preserve the isomorphism of games, so we define what we call a proper representation, in a way that two isomorphic games lie in the same proper representation. In doing so, the situation becomes more manageable to enumerate.

The concept of parametrization is widely used in the literature to understand the structure of simple games. For example, in \cite{CaFr96}, non-isomorphic complete simple games were parameterized in terms of a matrix and a vector that satisfies some constraints. This result allowed to enumerate complete simple games with two equivalence classes in \cite{FrMoRo2012} and later in \cite{KuTau2013}. The results in these two articles were revisited in \cite{FrKu13} where null and veto players were added, achieving up to four different types of players and finding also enumeration formulas. Similarly, in \cite{FrSa21DAM} a parametrization for non-isomorphic incomplete bipartite simple games was defined, and a closed formula for the size of some subclasses of these functions was found in the same article, concretely the one with minimum and maximum possible values of the amount of minimal winning vectors. Finally, in \cite{KuSa23DAM} a parameterization for non-isomorphic simple games was found and the remaining cases of the number of bipartite simple games left in \cite{FrSa21DAM} were enumerated. In this paper, we reuse the parameterization done in Theorem $1$ of \cite{KuSa23DAM} to compute the number of non-isomorphic $t$-partite simple games with one minimal winning vector. The obtained enumeration does not appear, at this time, in the OEIS.  

Furthermore, we compute the dimension of this class of simple games, that is, the minimum number of weighted games such that its intersection recovers the original simple game. The concept of weightedness is widely studied in the literature. See, for example, \cite{TaZw1992} for a characterization of weighted games in terms of trades, \cite{FrMo2019} for results related to minimum integer representations, and \cite{FrKu2014} for some enumerations about this class. Note that the concept of weightedness can be rephrased in the context where players can vote among ordered inputs and anonymity is required (see \cite{FrPo21}). Moreover, see \cite{FrPu08} for the computation of the dimension of a subclass of complete simple games. In that article the matrix parametrization does not correspond to the same done in this paper since the one done there uses the total desirability order from completeness property. Also, see \cite{Fr04,KuNa16} for examples of dimension of a simple game related to European Council voting rules.

The structure of the paper is as follows. Section \ref{S:preliminaries} for basic concepts and notation. Section \ref{S:Enum} contains the main results of the paper, that is, the enumeration formulas. Section \ref{S:dim} contains the results on the dimensions of the games studied in previous sections, and Section \ref{S:conclusion} concludes.

\section{Preliminaires}\label{S:preliminaries}

 Given a set of players (or voters) $N=\{1,2,3,\dots,n\}$, denote any subset of players a \textit{coalition}, $S\subseteq N$, and denote $2^N=\{S \ | \ S\subseteq N\}$ as the set of all coalitions. We define a simple game as a pair $(N,v)$ where 
$v \colon 2^N \to \{0,1\}$, satisfying (i) $v(\emptyset)=0$, (ii) $v(N)=1$ and (iii) for any coalitions $S,T$ such that $S\subseteq T$ we have $v(S)\leq v(T)$  (\textit{monotonicity}).

We say that $\mathcal{W}=\{S \ | \ v(S)=1\}$ is the set of \textit{winning coalitions} and $\mathcal{L}=\{S \ | \ v(S)=0\}$ is the set of \textit{losing coalitions}. As usual, a \emph{minimal winning} coalition is a winning coalition all of whose proper subsets are losing, it is $\mathcal{W}^m=\{W\in\mathcal{W} \ | \ v(T)=0 \text{ for any } T\subsetneq S\}$. Similarly we define the set of maximal losing coalitions $\mathcal{L}^M=\{S\in\mathcal{L} \ | \ v(T)=1 \text{ for any } S\subsetneq  T\}$.  Any simple game $(N,v)$ is determined by its set of players and its set of winning coalitions $(N,\mathcal{W})$ and, due to monotonicity, it is also determined by $(N,\mathcal{W}^m)$.

Two simple games $(N,W)$ and $(N',W')$ are \emph{isomorphic} if there exists a one-to-one correspondence $\varphi:N \rightarrow N'$ such that $S \in W$ \emph{if and only if} $\varphi(S) \in W'$; $\varphi$ is called and \emph{isomorphism} of simple games. For an extensive introduction to simple games see \cite{taylor_zwicker_book}.

Let $(N,W)$ be a simple game. Let $W_a = \{S \in W \, : \, a \in S \}$, $\tau_{ab}:N \rightarrow N$ denotes the transposition of players $a,b \in N$. The \emph{desirability relation}, introduced by Isbell in \cite{Is58}, is the binary relation $\succsim$ on $N$:
$a\succsim b$ if and only if  $\tau_{ab}(W_b)\subseteq W_a$
and, say that $a$ is \emph{at least as desirable as} $b$. The relation $\succsim$ is a preorder. The \emph{equi-desirability relation}, is the equivalence relation $\approx$ on $N$:
$a\approx b$  if and only if $a\succsim b$ and $b \succsim a$.
The preorder $\succsim$ induces an ordering $ \geqslant $ in the quotient set $N / \approx$ of equi-desirable classes, $N_1, N_2, \dots,N_t$.
Hence,   $N_p \geqslant N_q$  if and only if $a\succsim b$ for any $a\in N_p$ and any $b \in N_q$.

By $N_1,\dots,N_t$ we denote the equivalence classes of players, i.e., $t$ is the number of equivalence classes of players. For each coalition $S\subseteq N$ we define a corresponding 
coalition vector $\overline{s}=\left(\left|S\cap N_1\right|,\dots,\left|S\cap N_1\right|\right)\in\N^t$. For each coalition vector $m$, there can be several corresponding 
coalitions. However, all those coalitions have the same $v$-value, so that we speak of winning, minimal winning, losing, and maximal losing vectors. Let $r$ be the number
of different minimal winning vectors of $v$. For the special coalition $N$ we use the notation $\overline{n}=\left(n_1,\dots,n_t\right)$.

 As each simple game 
$v$ is uniquely characterized by its set of minimal winning coalitions and the player set $N$ it is also uniquely described by its set of minimal winning vectors
and $N_1,\dots,N_t$ (see Theorem $1$ in \cite{KuSa23DAM}). This characterisation is useful to enumerate non-isomorphic simple games since two isomorphic simple games lie in the same pair $(\overline{n},\mathcal{M})$.

Hence, instead of a set of minimal winning vectors we consider a matrix representation $\mathcal{M}=\left(m_{i,j}\right)\in\N^{r\times t}$ whose
$r$ rows $\overline{m}^1,\dots,\overline{m}^r$ are the minimal winning vectors. W.l.o.g.\ we assume that $\overline{m}^1\succ_{\operatorname{lex}}\overline{m}^2
\succ_{\operatorname{lex}}\dots\succ_{\operatorname{lex}}\overline{m}^r$, where $\succ_{\operatorname{lex}}$ denotes the lexicographical ordering of vectors in 
$\R^t\supset\N^t$. 

\begin{example}
  The pair consisting of $\overline{n}=(3,4,4)$ and $\mathcal{M}=\begin{pmatrix} 1&2&2\end{pmatrix}$ uniquely characterizes a simple game $v$ with $n=11$ players, $t=3$ equivalence 
  classes of players, and $r=1$ minimal winning vectors. We may assume that the players are numbered in such a way such that we have $N_1=\{1,2,3\}$, $N_2=\{4,5,6,7\}$, and $N_3=\{8,9,10,11\}$.
\end{example}

For any permutation $\pi$ in the symmetric group on $t$ elements, we denote by $\left(\overline{n},\mathcal{M}\right)^\pi=\left(\overline{n}^v,\mathcal{M}^v\right)$ the action of $\pi$ on the columns of $\overline{n}$ and $\mathcal{M}$, that is, we obtain a symmetric simple game. For $\pi=\begin{pmatrix}1&2\end{pmatrix}$ we obtain $\overline{n}^\pi=(4,3,4)$ and $\mathcal{M}^\pi=
\begin{pmatrix}2&1&2\end{pmatrix}$. For $\pi=\begin{pmatrix}2&3\end{pmatrix}$ we obtain $\overline{n}^\pi=(3,4,4)$ and $\mathcal{M}^\pi=
\begin{pmatrix}1&2&2\end{pmatrix}$.

The previous example induces the idea to avoid different pairs $(\overline{n},\mathcal{M})$ representing isomorphic games. The way to deal with it is to establish an order (lexicographic) among the matrices $\mathcal{M}$ through the next definition, which will be used in the parameterization in Definition \ref{def:conditions_sg}.

\begin{definition}[Definition 8 in \cite{KuSa23DAM}]
  Let $X=\left(x_{i,j}\right)\in\mathbb{N}^{r\times t}$ and $Y=\left(y_{i,j}\right)\in\mathbb{N}^{r\times t}$. We write $X\le Y$ (or $Y\ge X$) iff $\widehat{x}\le \widehat{y}$ (or 
  $\widehat{y}\ge \widehat{x}$), where
  $$
    \widehat{x}=\left(x_{1,1},\dots,x_{r,1},x_{1,2},\dots,x_{r,2},\dots,x_{1,t},\dots,x_{r,t}\right)\in\mathbb{N}^{rt}
  $$ 
  and
  $$
    \widehat{y}=\left(y_{1,1},\dots,y_{r,1},y_{1,2},\dots,y_{r,2},\dots,y_{1,t},\dots,y_{r,t}\right)\in\mathbb{N}^{rt}.
  $$ 
  In words we say that $X$ is \emph{lexicographically at most as large} as $Y$. We abbreviate the cases when $X\le Y$ and $X\neq Y$ by $X<Y$. The relation is called \emph{lexicographically larger} (or \emph{lexicographically smaller}).  
\end{definition}

In Theorem $1$ of \cite{KuSa23DAM} a parameterization of non-isomorphic simple games is stated in terms of a matrix and a vector deduced from a set of proprieties set in Definition $9$ in the same article. Here, we rewrite the definition restricted in the context of one minimal winning vector, which is the topic we deal in this paper.

\begin{definition}\label{def:conditions_sg} For a given simple game $v$ with $t$ equivalence classes of
players and one minimal winning vector $m$, i.e. $r=1$, let $\overline{n}\in\mathbb{N}_{>0}^t$, $\mathcal{M}\in \mathbb{N}^{t}$ invariants as defined above, where $t\in\mathbb{N}_{>0}$. We say that
$(\overline{n},\mathcal{M})$ is a \emph{proper representation} of $v$ if the following
conditions are satisfied:

  \begin{enumerate}
    \item[(a)] $\overline{n}_1\ge\overline{n}_2\ge\dots\ge\overline{n}_t>0$, $\sum_{i=1}^t \overline{n}_i=n$;
    \item[(b)]  $0\leq m_i\leq n_i$ \ for all $1\leq i\leq t$,
    \item[(c)] $\mathcal{M}\ge \mathcal{M}^\pi$ for every permutation $\pi$ of $\{1,\dots,t\}$ with $\overline{n}=\overline{n}^\pi$.           
  \end{enumerate}
\end{definition}

\begin{remark}    
Thus, Theorem $1$ in \cite{KuSa23DAM} applied in the one minimal vector context, says that the number of solutions $(\overline{n},\mathcal{M})$ of the inequalities of Definition \ref{def:conditions_sg} coincides with the number of non-isomorphic simple games with one minimal winning vector.
\end{remark}

We denote the set of pairs $(\overline{n},\mathcal{M})$ satisfying the conditions of Definition \ref{def:conditions_sg} by $\mathcal{SG}(n,t,r)$, and $SG(n,t,r)$ by its cardinality. Analogously, we denote $\mathcal{SG}^{\neg v, \neg n}(n,t,r)$ by the subset of games in $\mathcal{SG}(n,t,r)$ not containing neither nulls nor veto players, and $SG^{\neg v, \neg n}(n,t,r)$ its cardinality.  As mentioned, the results in this paper involve the sets with $r=1$.

The goal of this article, that we approach in the next section, is to count how many of these pairs $(\overline{n},\mathcal{M})$ of Definition \ref{def:conditions_sg} there are. Before that, two last definitions, that will be used in the enumeration results.

\begin{definition}[Cauchy product]\label{Def:cauchy_product} 
Now we point out the form of the coefficients of series obtained as the product of two series. Let $a(x)=\sum_{n\geq 0}a_nx^n$ and  $b(x)=\sum_{n\geq 0}b_nx^n$. If $c(x)=a(x)b(x)$, then $c_n$ can be obtained in

\begin{equation*}
    a(x)b(x)=\sum_{n\geq 0}\sum_{k=0}^n a_k b_{n-k}x^n
\end{equation*}
so
\begin{equation*}
    c_n=\sum_{k=0}^n a_k b_{n-k},
\end{equation*}
i.e. the so-called Cauchy product. 
\end{definition}

Later we will compute the Cauchy product of $t$ series so, in order to make the notation more manageable, we introduce the discrete convolution, which is defined for an infinite series over any $n\in\mathbb{Z}$, but it is only applied to the context with coefficients different from zero when $n\in\mathbb{N}\cup\{0\}$.

\begin{definition}[Discrete convolution]\label{Def:discrete_convolution}
    Given two functions $f(n),g(n)$ we define the (discrete) convolution as
    \begin{equation*}
        (f * g)(n)=\sum_{k=-\infty }^{+\infty } f(k)g(n-k).
    \end{equation*}

However, if the functions $f,g$ are different from zero only when its variable takes values between $0$ and $n$, we can rephrase the definition to:
    \begin{equation*}
        (f * g)(n)=\sum_{k=0}^{n} f(k)g(n-k).
    \end{equation*}

Recall convolution is commutative. Additionally, given $f_1,\dots,f_t$ we denote 
\begin{equation*}
    \Conv_{i=1}^t f_i=f_1*\dotsc* f_t.
\end{equation*}
We recall that the sequence $1,0,0,0,\dots$ is the neutral element of the (discrete) convolution. These sequences, $a_n$, can be represented by an infinity sum of powers of $x$, $\sum_{n\geq 0}a_nx^n$. Furthermore, we write the sequences involved in this article in terms of generating functions, $f(x)$. For example, 
\begin{equation*}
    f(x)=\frac{1}{1-x}=\sum_{n\geq 0} x^n=1+x+x^2+x^3+x^4+\dots
\end{equation*}
Note the equality holds since $(1-x)\sum_{n\geq 0} x^n=\sum_{n\geq 0} x^n-\sum_{n\geq 1} x^n=1$.

\end{definition}

\section{Enumeration of simple games with minimum}\label{S:Enum}

When the number of rows of the matrix, i.e. minimal winning vectors, is just one, we have enumerated how many simple games of this kind there are. Concretely, Theorem \ref{thm:enum} is the main result of this paper. From now on, we call such games simple games with minimum.

\begin{definition}\label{def:sg_with_sym}
  By $\SGs(n,t,r)$ we denote the number of different pairs $\left(\overline{n},\mathcal{M}\right)$, where $\overline{n}=\left(n_1,\dots,n_t\right)\in \N_{>0}^t$ with $\sum_{i=1}^t n_i=n$ and $\mathcal{M}\in\N^{r\times t}$, representing a simple game with $n$ players, $t$ equivalence classes of players ($N_i$ with cardinality $n_i$), and $r$ minimal winning vectors\footnote{In the literature the rows of $\mathcal{M}$ can be called minimal winning vectors.}, (specified by the rows of $\mathcal{M}$; assuming a decreasing lexicographical order of the rows of $\mathcal{M}$). 
\end{definition}

In other words, the difference between $SG(n,t,1)$ introduced in Definition \ref{def:conditions_sg} and,  $\SGs(n,t,r)$ in Definition \ref{def:sg_with_sym} is that $SG(n,t,r)$ counts after factoring out symmetry.

\begin{remark}
    Focusing in the case with only one equivalence class of players, we need to count the possible integer solutions of 
    \begin{equation*}
        1\leq m \leq n.
    \end{equation*}
    Moreover, if we want to exclude veto players, i.e. m=n, we get
    \begin{equation*}
        1\leq m \leq n-1.
    \end{equation*} which has $n-1$ solutions for $n\geq 1$. Furthermore,  we point out that null players are not allowed with only one class of players. Note that the sequence $0,0,1,2,3,4,5,\dots$ corresponds to the coefficients of the generating function $g(x)=\dfrac{x^2}{(1-x)^2}=x^2+2x^3+3x^4+4x^5+\dots$
\end{remark}

In order to generalize the previous result (i.e., excluding for the moment vetoers and null players) into the $t$-partite case we need to consider the integer solutions of the following system of inequalities.

    \begin{align}\label{eq:ineq}
    \begin{split}
        1\leq\hspace{0.1cm} & m_1 \leq n_1-1 \\
        1\leq\hspace{0.1cm} & m_2 \leq n_2-1 \\
        & \vdots \\
        1\leq\hspace{0.1cm} & m_t \leq n_t-1 \\
    \end{split}
    \end{align}
subject to $\sum_{i=1}^t n_i=n$.

The first step consists of counting all the possibilities for the vector $\overline{n}$ there are. Then count how many games there are for each possible $\overline{n}$, i.e., how many matrices $\mathcal{M}$ there are, and finally remove some isomorphisms. Additionally, we need to point out that by definition, for the case of one type of player, it is not possible to have all players being null or all players being vetoers. This kind of player need special treatment in the $t$-partite case.

Observe that for $t=1$ we need to have $r=1$. Next, we want to study the situation where $r=1$ (and $t>1$). Here we have $\mathcal{M}=\left(\overline{m}^1\right)=\left(m_{1,1},\dots,m_{1,t}\right)$, or $\left(m_{1},\dots,m_{t}\right)$ for simplicity. 
We observe that the players in equivalence class $N_i$ are null players if and only if $m_{i}=0$. Similarly, the players in equivalence class $N_i$ are veto players if and only if $m_{i}=n_i$.

\begin{definition}


For games with neither null nor veto players we use the notations 
$\SGs^{\neg n,\neg v}(n,t,r)$ and $\SG^{\neg n,\neg v}(n,t,r)$. 

  For each $(t,r)\in\N_{>0}^2$ we define the generation functions
  \begin{gather*}
    \hat{f}(x;t,r):=\sum_{n\in\N} \SGs(n,t,r)\cdot x^n,\quad\text{ }\quad f(x;t,r):=\sum_{n\in\N} \SG(n,t,r)\cdot x^n, \\
    \hat{f}^{\neg n,\neg v}(x;t,r):=\sum_{n\in\N} \SGs^{\neg n,\neg v}(n,t,r)\cdot x^n\quad\text{and}\quad f^{\neg n,\neg v}(x;t,r):=\sum_{n\in\N} \SG^{\neg n,\neg v}(n,t,r)\cdot x^n,
  \end{gather*}
\end{definition}

\begin{proposition}\label{prop:bound}$\,$
\begin{enumerate}[label=(\alph*)]
    \item If $v$ is a simple game with $n$ players, $t$ equivalence classes of players, without null or veto players. Then, we have $t\le \left\lfloor\tfrac{n}{2}\right\rfloor$.
    \item If $v$ is a simple game with $n$ players and,  $t$ equivalence classes of players. Then, we have $t\le \left\lfloor\tfrac{n}{2}+1\right\rfloor$.
  \end{enumerate}
\end{proposition}

\begin{proof}

$(a)$ The set of inequalities in Equation (\ref{eq:ineq}) implies $n_i>1$ for any $i=1,\dots,t$, which automatically imposes a bound on the maximum value of $t$. For $n$ even, the maximum possible value of $t$ is achieved when all $n_i$ are equal to $2$, hence $t=\frac{n}{2}$, and for $n$ odd, the maximum appears when all $n_i$ but one are equal to two and the other is equal to three, hence $t=\frac{n-1}{2}$. In other words $t\le \left\lfloor\tfrac{n}{2}\right\rfloor$.

\medskip

$(b)$ A null player or a veto player can belong to a class of only one player, hence, similarly than in $(a)$, there is room for one more space since we can split one class of two players into a single veto player and a single null player, thus $t\le \left\lfloor\tfrac{n}{2}+1\right\rfloor$.

Finally note that for any of the vectors $\overline{n}$ mentioned above, there exists one game, consisting in $m_i=1$ for any class of players except the class of nulls (if exists) which has assigned a $0$.

\end{proof}

\begin{proposition}\label{prop:integer_sol}
    Given $t$ and $n$, the number of integer solutions of     \begin{align*}
        1\leq\hspace{0.1cm} & m_1 \leq n_1-1 \\
        1\leq\hspace{0.1cm} & m_2 \leq n_2-1 \\
        & \vdots \\
        1\leq\hspace{0.1cm} & m_t \leq n_t-1 \\
    \end{align*}
subject to $\sum_{i=1}^t n_i=n$, corresponds to the $n$-th coefficient of the generating function $g(x)^t$. These coefficients can be written as

   $$
\widehat{SG}^{\neg v, \neg n}(n,t,1) =  \dbinom{n-1}{2t-1} 
$$ or in terms of generating functions
\begin{equation}\label{eq:g(x)^t}
    \hat{f}^{\neg n,\neg v}(x;t,1)=g(x)^t=\sum_{n\geq 0} \dbinom{n-1}{2t-1} x^n
\end{equation}
Note that, by binomial's definition, it is $0$ when $n<2t$.

\end{proposition}

Before the proof, we recall a binomial property, known as Chu-Vandermonde identity, which applies for any integers $j,k,r$ such that $0\leq j\leq k\leq n$,
\begin{equation}\label{Eq:chuvandermonde}
    \sum_{m=0}^r\binom{m}{j}\binom{r-m}{k-j}=\binom{r+1}{k+1}.
\end{equation}

\begin{proof}
    By induction on $t$, we know that for $t=1$, it holds $\binom{n-1}{n-2}=n-1$, as expected. Suppose it is true for $t-1$. Note 
    \begin{equation*}
        \widehat{SG}^{\neg v, \neg n}(n,t,1)=\sum_{i=1}^{n-2t+2} \widehat{SG}^{\neg v, \neg n}(n-i-1,t-1,1)\cdot (i-1).
    \end{equation*}
    The previous formula is true due to the following argument: it is not possible to get the games of $n$ players and $t$ classes from games with $n-1$ players and $t-1$ classes since having $n_t=1$ would imply to have null or veto players, and we are not counting them here. The way to have $t$ equivalent classes from $t-1$ classes is: first to consider the case with $t-1$ classes of $n-2$ players and adding one new class with two players, then consider all cases with $t-1$ classes of $n-3$ players and adding a class of three players (i.e. $n_t=3$), and so on. 
    
    We can do this process until the maximum size of $n_t=n-2t+2$, since in the previous $t-1$ classes there should be at least two player on each class. Also, pointing out that adding a class of size $n_t$ generates $n_t-1$ options in the new restriction $1\leq m_t<n_t$, we need to multiply $\widehat{SG}^{\neg v, \neg n}(n-i,t-1,1)$ by $i-1$.  

    Now, using induction for $t\geq 2$,
    \begin{align*}
        \widehat{SG}^{\neg v, \neg n}(n,t,1)&=\sum_{i=1}^{n-2t+2} \binom{n-i-1}{2(t-1)-1}\cdot (i-1)\\&=\sum_{i=1}^{n-2t+2} \binom{n-i-1}{2t-3}\cdot (i-1)\\&=\sum_{i=1}^{n-2} \binom{n-i-1}{2t-3}\cdot (i-1)\\
        \end{align*}
since $\binom{a}{b}=0$ when $a<b$. Now applying the Chu-Vandermonde identity in Equation (\ref{Eq:chuvandermonde}) with $j=1$, $m=i-1$, $k=2t-2$ and $r=n-2$ we obtain
        \begin{align*}
        &\dbinom{n-1}{2t-1}
    \end{align*}
as expected. Also, for $n<2t$ there are no terms in the summation so $\widehat{SG}^{\neg v, \neg n}(n,t,1)=0$.    
\end{proof}

Note that the integer solutions of Proposition \ref{prop:integer_sol} still contain isomorphic games.

\begin{example}
    The games $(\overline{n}_1,\mathcal{M}_1)$ and $(\overline{n}_2,\mathcal{M}_2)$ are isomorpic where
\vspace{-1cm}\begin{multicols}{2}
\begin{align*}
\overline{n}_1&=(2,3,3,4)\\             
\mathcal{M}_1&=\begin{pmatrix}
1 & 1 & 2 & 1
\end{pmatrix}
\end{align*}
\bigskip
\begin{align*}
\overline{n}_2&=(2,3,4,3)\\             
\mathcal{M}_2&=\begin{pmatrix}
1 & 1 & 1 & 2
\end{pmatrix}
\end{align*}
\end{multicols}
{\noindent and both are integer solutions of Proposition \ref{prop:integer_sol}. In fact, there are in total $24$ solutions of Equation (\ref{eq:ineq}) that lie in the same proper representation. The number of proper representations has to be computed using symmetric groups and cyclic indices. In other words, while all $24$ solutions satisfy Condition $(b)$ in Definition \ref{def:conditions_sg}, only two of them satisfy Condition $(a)$, and only one of these two also satisfies $(c)$. See Example $3$ in \cite{KuSa23DAM} for more details.}
\end{example}

In order to count only non-isomorphic games we should apply a similar argument than the one in Lemma $7$ in \cite{KuSa23DAM} but for an arbitrary number of equivalence classes of players, $t$. To do that, we apply Pólya's Enumeration Theorem for generating functions (see \cite{Po37,PoRe87,Re27}). 

\begin{definition}
  Generically speaking, the cycle index $Z(G)$ of a permutation group $G$ is the average of the cycle index monomials of all the permutations $g\in G$, i.e.\ 
  \begin{equation} 
    Z(G)=\frac{1}{|G|}\cdot  \sum_{g\in G}\prod_{k=1}^n a_k^{j_k(g)},
  \end{equation}
  where $n$ is the degree of $G$ and $j_k(g)$ encodes the cycle structure of $g$, (see e.g. \cite{wiki:cyclic_index}).
\end{definition}

\begin{example} For example, reviewing the next symmetric groups, we have:
  \begin{eqnarray*}
    Z(S_2) &=& \frac{a_1^2+a_2}{2}\\ 
    Z(S_3) &=& \frac{a_1^3+3a_1a_2+2a_3}{6}\\
    Z(S_4) &=& \frac{a_1^4+8a_1a_3+3a_2^2+6a_1^2a_2+6a_4}{24}.\\
  \end{eqnarray*}
\end{example}

Considering the symmetric group $S_t$ where each element of the group consists in one equivalence class, and using that the enumeration we are pursuing is defined by $g(t)$ for a single equivalence class, Pólya's Enumeration Theorem states that the number of \emph{non-isomorphic} simple games with one minimal winning profile, $n$ players and $t$ classes of players and neither veto nor null players, denoted by $SG^{\neg v, \neg n}(n,t,1)$, is given by the coefficients of the following generating function:
\begin{equation}\label{eq:symmetric}
   f^{\neg n,\neg v}(x;t,1)=\sum_{j_1+2j_2+3j_3+\dots+tj_t=t}\frac{1}{\prod_{k=1}^t(k^{j_k}j_k!)}\prod_{k=1}^t g(x^k)^{j_k},
\end{equation}
or recursively
    \begin{equation}\label{eq:recursion}
        f^{\neg n,\neg v}(x;t,1)=\frac{1}{t}\sum_{k=1}^{t}\left(g(x^k)\cdot f^{\neg n,\neg v}(x;t-k,1)\right).
    \end{equation}

Note that another notation in the literature for the generating function in Equation (\ref{eq:symmetric}) is $Z_{S_t}(g(x),g(x^2),g(x^3),\dots)$.

Pólya's Enumeration Theorem applies for constants, polynomials and generating functions. The computation of this formula result into another generating function that give us the coefficients of an enumeration we are looking for (Theorem \ref{thm:enum_neq_v_n}). As an example, we will use this equality to compute the number of non-isomorphic games of nine players and three classes of players, so $n=9$, i.e. we will look for the term $x^9$ of the resulting generating function when $t=3$. 

However, first we need to understand the terms $g(x^k)^{j_k}$, so we rephrase Equation (\ref{eq:g(x)^t}) to announce Theorem \ref{thm:enum_neq_v_n}.

\begin{proposition}\label{prop:g^a^b}
    Given integers $a>1$ and $b>0$, we have
    \begin{equation*}
        g(x^a)^b=\sum_{n\geq 0}\binom{\left(\frac{n}{a}-1\right)\cdot\mathbbm{1}_{\{n\equiv 0 \mod{a}\}}}{2b-1}x^n
    \end{equation*}
    where $\mathbbm{1}_{\{n\equiv 0 \mod{a}\}}$ states for the characteristic function that takes value $1$ when $n\equiv 0 \mod{a}$, and $0$ otherwise. 
\end{proposition}

\begin{proof}
    Note when $a=1$ the formula matches with Proposition $\ref{prop:integer_sol}$. Also, to consider the power of a variable keeps the coefficients the same but translated in therms of this power. Considering $y=x^a$,
    \begin{equation*}
        g(y)^b=\dfrac{y^{2t}}{(1-y^2)^t}=\sum_{n'\geq 0}\dbinom{n'-1}{2t-1} y^{n'},    \end{equation*}
        and undoing the change, we obtain
    \begin{equation*}
        g(x^a)=\sum_{n\geq 0}(n'-1) x^{an'},    \end{equation*}
        hence $n=an'$ and $n'=\frac{n}{a}$, so the coefficient is $\dbinom{\frac{n}{a}-1}{2t-1}$ only in the multiples of $a$, and zero otherwise. In other words,
        \begin{equation*}
            g(x^a)^b=\sum_{n\geq 0}\binom{\left(\frac{n}{a}-1\right)\cdot\mathbbm{1}_{\{n\equiv 0 \mod{a}\}}}{2b-1}x^n
        \end{equation*}        
\end{proof}
\vspace{-0.2cm}
Also, $g(x^a)^0=1$, i.e. the sequence $1,0,0,0,\dots$, which plays the role of the neutral element in the convolution. Furthermore, note that the fraction $\frac{n}{a}-1$ is an integer when the characteristic function takes value $1$.

Now we can give an example of Pólya's Enumeration Theorem for a fixed number of players instead of all generating functions. Note first that $f^{\neg n,\neg v}(x;0,1)=1$ and $f^{\neg n,\neg v}(x;1,1)=g(x)$. 
\begin{example}\label{ex:9players}
Using the previous results, we find the number of simple games with one minimal winning vector, neither veto nor null players, $9$ players and $3$ equivalence classes of players. To do it, we get the coefficient of the term $x^9$ of the following generating function. However, to simplify the computations, we can replace $g(x)$ by only $x^2+2x^3+\dots+6x^7$. Using the recursion formula in Equation (\ref{eq:recursion}), we get

    \begin{gather*}
    f^{\neg n,\neg v}(x;3,1)=\frac{1}{3}\sum_{k=1}^{3}\left(g(x^k)\cdot f^{\neg n,\neg v}(x;3-k,1)\right),
\end{gather*}
so we need to compute before $f^{\neg n,\neg v}(x;2,1)$. 
\begin{align*}
     f^{\neg n,\neg v}(x;2,1)&=\frac{1}{2}\left((x^2+2x^3+\dots+6x^7)\cdot f^{\neg n,\neg v}(x;1,1)+(x^4+2x^6+\dots+6x^{14})\cdot f^{\neg n,\neg v}(x;0,1)\right),\\
     &=\frac{1}{2}\left((x^2+2x^3+\dots+6x^7)\cdot (x^2+2x^3+\dots+6x^7)+(x^4+2x^6+\dots+6x^{14})\right)
\end{align*}
which is equal to some coefficients of degree bigger than $7$ and the term
\begin{equation*}
     x^4+ 2 x^5+ 6 x^6 +10 x^7,
\end{equation*}
which is the one we need to compute the remaining term:
\begin{gather*}
    f^{\neg n,\neg v}(x;3,1) =\frac{1}{3}\left((x^2+2x^3+\dots+6x^7)\cdot(x^4+ 2 x^5+ 6 x^6 +10 x^7+\dots)\right)+\\+\left((x^4+2x^6+\dots+6x^{14})\cdot (x^2+2x^3+\dots+6x^7)\right)+\left((x^6+2x^9+\dots+6x^{21})\cdot 1\right)=\\
    =x^6+2 x^7 +6 x^8 +14 x^9 +\dots
\end{gather*}
So there are $14$ non-isomorphic simple games with one minimum wining vector, $9$ players and $3$ classes of players without null or veto players. Concretely, the following ones divided into $3$ possible values of the vector $\overline{n}$:

\begin{multicols}{3}
\noindent  
\begin{align*}
\overline{n}_1&=(5,2,2)\\    
\mathcal{M}_1&=\begin{pmatrix}
4 & 1 & 1 
\end{pmatrix}\\             
\mathcal{M}_2&=\begin{pmatrix}
3 & 1 & 1 
\end{pmatrix}\\             
\mathcal{M}_3&=\begin{pmatrix}
2 & 1 & 1 
\end{pmatrix}\\             
\mathcal{M}_4&=\begin{pmatrix}
1 & 1 & 1 
\end{pmatrix}
\end{align*}
\columnbreak
\begin{align*}
\overline{n}_2&=(4,3,2)\\
\mathcal{M}_5&=\begin{pmatrix}
3 & 2 & 1 
\end{pmatrix}\\             
\mathcal{M}_6&=\begin{pmatrix}
2 & 2 & 1 
\end{pmatrix}\\             
\mathcal{M}_7&=\begin{pmatrix}
1 & 2 & 1 
\end{pmatrix}\\             
\mathcal{M}_8&=\begin{pmatrix}
3 & 1 & 1 
\end{pmatrix}\\             
\mathcal{M}_9&=\begin{pmatrix}
2 & 1 & 1 
\end{pmatrix}\\             
\mathcal{M}_{10}&=\begin{pmatrix}
1 & 1 & 1 
\end{pmatrix}
\end{align*}
\columnbreak
\begin{align*}
\overline{n}_3&=(3,3,3)\\             
\mathcal{M}_{11}&=\begin{pmatrix}
2 & 2 & 2 
\end{pmatrix}\\             
\mathcal{M}_{12}&=\begin{pmatrix}
2 & 2 & 1 
\end{pmatrix}\\             
\mathcal{M}_{13}&=\begin{pmatrix}
2 & 1 & 1 
\end{pmatrix}\\             
\mathcal{M}_{14}&=\begin{pmatrix}
1 & 1 & 1 
\end{pmatrix}
\end{align*}
\end{multicols}
\end{example}
Now we are able to represent the general term $ SG^{\neg v,\neg n}(n,t,1)$ replacing the generic terms $g(x^k)^{j_k}$ of Equation (\ref{eq:recursion}) by the ones obtained in Proposition \ref{prop:g^a^b} and representing the multiplication of series as the convolution of its coefficients.

\begin{theorem}\label{thm:enum_neq_v_n}
    The number of non-isomorphic simple games with one minimal winning vector and neither veto nor null players is given by the following equation.
    \begin{equation*}    
    SG^{\neg v,\neg n}(n,t,1)=\sum_{j_1+2j_2+3j_3+\dots+tj_t=t}\frac{1}{\prod_{k=1}^t(k^{j_k}j_k!)}\Conv_{\substack{k=1\\
                  j_k\neq 0}}^t\binom{\frac{n-k}{k}\cdot\mathbbm{1}_{\{n\equiv 0 \mod{k}\}}}{2j_k-1}
\end{equation*}
where $\Conv$ denotes the convolution product of all values varying $k$ from $1$ to $t$. Note that the convolution applies to a function depending on $n$, where $k$ and $j_k$ are fixed.
\end{theorem}
\begin{proof}
Plugging the expression from Proposition \ref{prop:g^a^b} into Equation (\ref{eq:recursion}) we get

\begin{equation}
    f^{\neg n,\neg v}(x;t,1)=\sum_{j_1+2j_2+3j_3+\dots+tj_t=t}\frac{1}{\prod_{k=1}^t(k^{j_k}j_k!)}\prod_{\substack{k=1\\
                  j_k\neq 0}}^t \left( \sum_{n\geq 0}\binom{(\frac{n}{k}-1)\cdot\mathbbm{1}_{\{n\equiv 0 \mod{k}\}}}{2j_k-1}x^n \right).
\end{equation}
Note that we are excluding the terms $g(x^k)^0$ of the product since it is equivalent to multiply by $1$. Then, the Cauchy product of these series become

\begin{equation*}    
   f^{\neg n,\neg v}(x;t,1)=\sum_{j_1+2j_2+3j_3+\dots+tj_t=t}\frac{1}{\prod_{k=1}^t(k^{j_k}j_k!)}\left(\sum_{n\geq 0}\left(\Conv_{\substack{k=1\\
                  j_k\neq 0}}^t\binom{\frac{n-k}{k}\cdot\mathbbm{1}_{\{n\equiv 0 \mod{k}\}}}{2j_k-1}\right)x^n\right),
\end{equation*}
so, indeed, coefficient by coefficient we get
 \begin{equation*}    
    SG^{\neg v,\neg n}(n,t,1)=\sum_{j_1+2j_2+3j_3+\dots+tj_t=t}\frac{1}{\prod_{k=1}^t(k^{j_k}j_k!)}\Conv_{\substack{k=1\\
                  j_k\neq 0}}^t\binom{\frac{n-k}{k}\cdot\mathbbm{1}_{\{n\equiv 0 \mod{k}\}}}{2j_k-1}
\end{equation*}
as expected.
\end{proof}
We are using the convolution to simplify the notation but would be equivalent to represent it as a succession of summation as the Cauchy product does. Form the recursive formula of the symmetric group (Equation (\ref{eq:recursion})) we get the following corollary.

\begin{corollary}\label{cor:SG_neq_v_n}
 The number of non-isomorphic simple games with one minimal winning vector and neither veto nor null players is given by the following equation.

       \begin{equation*}
    SG^{\neg v,\neg n}(n,t,1)=\frac{1}{t}\sum_{l=1}^t \sum_{k=1}^{n} \left(\left(\left(\frac{k}{l}-1\right)\cdot\mathbbm{1}_{\{k\equiv 0 \mod{l}\}}\right)\cdot SG^{\neg v,\neg n}(n-k,t-l,1) \right)
\end{equation*}

\end{corollary}

For the proof, we just need to point out that, to multiply the two generating functions from Equation (\ref{eq:recursion}) we get a discrete convolution. Furthermore, the first of these two generating functions that we need to multiply for each $l$ corresponds to the one in Proposition \ref{prop:g^a^b} with $b=1$. Note that this binomial is $0$ when $n=0$, since it is outside the bounds where the binomial takes non zero values, so we can exclude this term from the summation. Additionally, the upper part of the binomial is always a positive integer since the characteristic function takes value zero when the fraction is not a positive integer. To simplify the calculations, we will denote this kind of fraction (the ones that end up being zero) by $\triangle$.

\begin{example}(Revisiting Example \ref{ex:9players}).
We will again compute the number of simple games with $9$ players and $3$ classes of players without nulls or vetoers using Corollary \ref{cor:SG_neq_v_n}. We start with

       \begin{equation*}
    SG^{\neg v,\neg n}(9,3,1)=\frac{1}{3}\sum_{l=1}^3 \sum_{k=1}^{9} \left(\left(\left(\frac{k}{l}-1\right)\cdot\mathbbm{1}_{\{k\equiv 0 \mod{l}\}}\right)\cdot SG^{\neg v,\neg n}(9-k,3-l,1) \right).
    \end{equation*}
Recall $SG^{\neg v,\neg n}(n,0,1)$ is $0$ if $n>0$ and $1$ for $n=0$. Also, $SG^{\neg v,\neg n}(n,1,1)=n-1$. Hence, first we need to compute $SG^{\neg v,\neg n}(n,2,1)$ for $n=0,\dots,8$. For example,  

       \begin{align*}
    SG^{\neg v,\neg n}(5,2,1)&=\frac{1}{2}\sum_{l=1}^2 \sum_{k=1}^{5} \left(\left(\left(\frac{k}{l}-1\right)\cdot\mathbbm{1}_{\{k\equiv 0 \mod{l}\}}\right)\cdot SG^{\neg v,\neg n}(5-k,2-l,1) \right)\\
    &=\frac{1}{2}( ((0\cdot 1)\cdot 3)+((1\cdot 1)\cdot 2)+((2\cdot 1)\cdot 1)+((3\cdot 1)\cdot 0)+((4\cdot 1)\cdot 0))+\\
    &+((\triangle\cdot 0)\cdot 0)+((0\cdot 1)\cdot 0)+((\triangle\cdot 0)\cdot 0)+((1\cdot 1)\cdot 0)+((\triangle\cdot 0)\cdot 1)) \\
    &=\frac{1}{2}(2+2)=2
\end{align*}

and repeating the process from the remaining values, we obtain Table \ref{table_SG_neq_v_n_5_2}.

\begin{table}[!h]
\begin{center}
\begin{tabular}{ | c | c | c | c |c | c | c | c |c | c | } 

  \hline
  n & 0 & 1 & 2 & 3 & 4 & 5 & 6 & 7 & 8   \\ 
  \hline
  $SG^{\neg v,\neg n}(n,2,1)$ &0 & 0 & 0 & 0 & 1 & 2 & 6 & 10 & 19    \\ 
  \hline
  \end{tabular}
\caption{Values of $SG^{\neg v,\neg n}(n,2,1)$ for $n=0,\dots,8$. \label{table_SG_neq_v_n_5_2}}
\end{center}
\end{table}
So we compute, 
       \begin{align*}
    SG^{\neg v,\neg n}(9,3,1)&=\frac{1}{3}\sum_{l=1}^3 \sum_{k=1}^{9} \left(\left(\left(\frac{k}{l}-1\right)\cdot\mathbbm{1}_{\{k\equiv 0 \mod{l}\}}\right)\cdot SG^{\neg v,\neg n}(9-k,3-l,1) \right)\\
    &=\frac{1}{3}(
    ((0\cdot 1)\cdot 19)+((1\cdot 1)\cdot 10)+((2\cdot 1)\cdot 6)+((3\cdot 1)\cdot 2)+((4\cdot 1)\cdot 1)+((5\cdot 1)\cdot 0)\\&+((6\cdot 1)\cdot 0)+((7\cdot 1)\cdot 0)+((8\cdot 1)\cdot 0)\\&+ ((\triangle\cdot 0)\cdot 7)+((0\cdot 1)\cdot 6)+((\triangle\cdot 0)\cdot 5)+((1\cdot 1)\cdot 4)+((\triangle\cdot 0)\cdot 3)+((2\cdot 1)\cdot 2)\\&+((\triangle\cdot 0)\cdot 1)+((3\cdot 1)\cdot 0)+((\triangle\cdot 0)\cdot 0)\\&+((\triangle\cdot 0)\cdot 0)+((\triangle\cdot 0)\cdot 0)+((0\cdot 1)\cdot 0)+((\triangle\cdot 0)\cdot 0)+((\triangle\cdot 0)\cdot 0)+((1\cdot 1)\cdot 0)\\&+((\triangle\cdot 0)\cdot 0)+((\triangle\cdot 0)\cdot 0)+((2\cdot 1)\cdot 1))\\
    &=\frac{1}{3}(10+12+6+4+4+4+2)=14
    \end{align*}
obtaining the expected value.
\end{example}

Next result considers all the possible ways to add veto and null players into a game without them, in order to obtain all simple games with minimal for a given $n$ and $t$.

\begin{theorem}\label{thm:enum} For $t> 2$,
    \begin{equation*}
        SG(n,t,1)=SG^{\neg v, \neg n}(n,t,1)+2\sum_{i=1}^{n-2} SG^{\neg v, \neg n}(n-i,t-1,1)+\sum_{i=2}^{n-2}(i-1)\cdot SG^{\neg v, \neg n}(n-i,t-2,1)
    \end{equation*}

    For $t=2$
    \begin{equation*}
        SG(n,2,1)=SG^{\neg v, \neg n}(n,2,1)+2\sum_{i=1}^{n-2} SG^{\neg v, \neg n}(n-i,1,1)+n-1.
    \end{equation*}
\end{theorem}

\begin{proof} \\
    Case $t>2$. A game with one minimal winning coalition and $n$ players and $t>2$ equivalent classes of players that can contain null or veto players comes from one of these four categories
    \begin{enumerate}
        \item It does not contain neither null nor veto players. There are $SG^{\neg v, \neg n}(n,t,1)$ of such non-isomorphic games.
        \item It contains $i$ null players for $i\in\{1,\dots,n-2\}$.  There are $\sum_{i=1}^{n-2} SG^{\neg v, \neg n}(n-i,t-1,1)$ non-equivalent games of this kind. Note the upper bound is $n-2$ since we have $t>2$ classes of equivalent players, and at most there are $n-2$ null players in one class and, at least one player in each of the $t-1\geq 2$ remaining classes.
        \item It contains $i$ veto players for $i\in\{1,\dots,n-2\}$. Analogously as in $b)$, there are $\sum_{i=1}^{n-2} SG^{\neg v, \neg n}(n-i,t-1,1)$ non-equivalent games.
        \item It contains a total number of $i$ ``nulls plus veto players", with at least one of each class. There are $i-1$ ways to make the choice. Once it is done, $n-i$ players remain into $t-2$ classes, so there are $\sum_{i=2}^{n-2}(i-1)\cdot SG^{\neg v, \neg n}(n-i,t-2,1)$ non-isomorphic games of this type.  
    \end{enumerate}
Case $t=2$. The arguments in $a)$, $b)$ and $c)$ are still valid. Regarding $d)$ we just need to count how many games with $n$ players with only veto and null players there are, with at least one of each kind, which are $n-1$.
    
\end{proof}

Note that $SG(n,2,1)$ coincides with Corollary $5.4$ of \cite{FrSa21DAM}. 
\begin{example}

Additionally in terms of generating functions, $f(x;t,1)$, we have
\begin{eqnarray*}
  f(x;2,1)&=&\frac{x^2\cdot\left(1+2x+x^2-2x^3\right)}{(1-x)^2(1-x^2)^2}\\ 
  &=& \frac{1}{2}\cdot\frac{x^4}{(1-x)^4}+\frac{1}{2}\cdot \frac{x^4}{(1-x^2)^2}+\frac{2x^3}{(1-x)^3}+\frac{x^2}{(1-x)^2},\\
  f(x;3,1)&=&\frac{x^4\cdot\left(1+4x+4x^2+2x^3-3x^4-2x^6\right)}{(1-x)^2(1-x^2)^2(1-x^3)^2}, \\
   f(x;4,1)&=&{\tiny\frac{x^6\cdot\left(1+4x+7x^2+8x^3+11x^4+6x^5+3x^6-2x^7-7x^8-6x^9-2x^{11}+x^{12}\right)}{(1-x)^2(1-x^2)^2(1-x^3)^2(1-x^4)^2},}
\end{eqnarray*}

\end{example}

Since Proposition \ref{prop:bound} proposes bounds on $t$, we have the following result.

\begin{theorem}

\begin{equation*}
    SG^{\neg v,\neg n}(n,1)=\sum_{i=1}^{\lfloor\frac{n}{2}\rfloor} SG^{\neg v,\neg n}(n,i,1)
\end{equation*}
and
\begin{equation*}
    SG(n,1)=\sum_{i=1}^{\lfloor\frac{n}{2}\rfloor+1} SG(n,i,1)
\end{equation*}
    
\end{theorem}

\begin{proof}
    From Proposition \ref{prop:bound} we just need to sum the terms satisfying the bounds.
\end{proof}

In Table~\ref{table_sg_with_minimum} we list the values of $\SG(n,t,1)$ for all $n\le 20$. Except for $\SG(n,1,1)=n$ none of the occurring
integer sequence is contained in the OEIS. 
    
\begin{table}[H]
  \begin{center}
    \begin{tabular}{rrrrrrrrrrrrr}
      \hline
      n/t & 1 &   2 &    3 &     4 &     5 &     6 &    7 &   8 &   9 & 10 & 11 & $SG(n,1)$ \\ 
      \hline
       1 &  1 &     &      &       &       &       &      &     &     &    &   &     1 \\ 
       2 &  2 &   1 &      &       &       &       &      &     &     &    &   &     3 \\
       3 &  3 &   4 &      &       &       &       &      &     &     &    &   &     7 \\
       4 &  4 &  10 &    1 &       &       &       &      &     &     &    &   &    15 \\
       5 &  5 &  18 &    6 &       &       &       &      &     &     &    &   &    29 \\
       6 &  6 &  31 &   17 &     1 &       &       &      &     &     &    &   &    55 \\
       7 &  7 &  46 &   40 &     6 &       &       &      &     &     &    &   &    99 \\
       8 &  8 &  68 &   79 &    20 &     1 &       &      &     &     &    &   &   176 \\
       9 &  9 &  92 &  146 &    52 &     6 &       &      &     &     &    &   &   305 \\
      10 & 10 & 125 &  244 &   122 &    20 &     1 &      &     &     &    &   &   522 \\
      11 & 11 & 160 &  392 &   252 &    56 &     6 &      &     &     &    &   &   877 \\
      12 & 12 & 206 &  598 &   485 &   139 &    20 &    1 &     &     &    &   &  1461 \\
      13 & 13 & 254 &  882 &   872 &   316 &    56 &    6 &     &     &    &   &  2399 \\
      14 & 14 & 315 & 1258 &  1494 &   659 &   144 &   20 &   1 &     &    &   &  3905 \\
      15 & 15 & 378 & 1756 &  2444 &  1298 &   338 &   56 &   6 &     &    &   &  6291 \\
      16 & 16 & 456 & 2387 &  3871 &  2416 &   744 &  144 &  20 &   1 &    &   & 10055 \\
      17 & 17 & 536 & 3192 &  5924 &  4314 &  1540 &  344 &  56 &   6 &    &   & 15929 \\ 
      18 & 18 & 633 & 4191 &  8844 &  7399 &  3042 &  771 & 144 &  20 &  1 &   & 25063 \\
      19 & 19 & 732 & 5424 & 12878 & 12286 &  5748 & 1646 & 344 &  56 &  6 &   & 39139 \\
      20 & 20 & 850 & 6921 & 18387 & 19791 & 10478 & 3352 & 778 & 144 & 20 & 1 & 60742 \\
      \hline
    \end{tabular}
    \caption{$\SG(n,t,1)$ for small parameters and its sum $SG(n,1)$ on the last column.}
    \label{table_sg_with_minimum}
  \end{center}
\end{table}

\section{Dimension of simple games with minimum}\label{S:dim}

When considering only one minimal winning vector, the concept of dimension becomes straightforward to compute. In this section, we show, for this context, the relation between the dimension and the number of equivalent classes of players $t$.

Recall a simple game $(N,v)$ is said to be \emph{weighted} if there exists a function $w:N\rightarrow\mathbb{R}$ and a quota $q\in\mathbb{R}$ that assigns to each player $i$ a \emph{weight} $w_i$ such that 
 \begin{equation*}
     \sum_{i\in W}w_i \geq q > \sum_{i\in L}w_i
 \end{equation*}
 for any $W\in\mathcal{W}^m$ and $L\in\mathcal{L}^M$.

 Additionally, given two simple games involving the same set of players, $(N,v_1)$ and $(N,v_2)$, we define, as usual, the game \emph{intersection} $(N,v)$ as $v(S)=1 \iff v_{1}(S)=1\text{ and }v_{2}(S)=1$. We say that a game $(N,v)$ has dimension $k$ if it can be represented as the intersection of $k$ weighted games but not as the intersection of $k-1$ weighted games.

\begin{theorem}
  Let $\overline{n}=\left(n_1,\dots,n_t\right)\in\N_{>0}^t$ and $\overline{m}=\left(m_1,\dots,m_t\right)\in\N^t$ uniquely characterize a simple game $v$ with minimum. If $v$ does neither contain
  null players nor veto players, then $\dim(v)=t$.
\end{theorem}
\begin{proof}
  Since $v$ does not contain veto or null players, we have $1\le m_i\le n_i-1$ for all $1\le i\le t$. Let $N_i$ denote the set of players corresponding to $n_i$. For each $1\le i\le t$ let $v^i$ denote the weighted game with quota $m_i$ where the players in $N_i$ have weight $1$ and all others have weight $0$. With this we have $v=\cap_{1\le i\le t} v^i$, so that $\dim(v)\le t$. 
  
  For the other direction we can assume $t\ge 2$. 
  For each index $1\le i\le t$ let $L_i\subseteq N$ such that $L_j=N_j$ for all $1\le j\le n$ with $j\neq i$ and $\left|L_i\cap N_i\right|=m_i-1$, so that $L_i$ is a losing coalition. For each 
  $1\le i<j\le t$ let $a\in \left(L_i\backslash L_j\right) \cap N_j$ and $b\in \left(L_j\backslash L_i\right) \cap N_i$. With this, $L_i\backslash \{a\}\cup\{b\}$ and $L_j\backslash \{b\}\cup\{a\}$ 
  are winning coalitions. Thus, $L_i$ and $L_j$ cannot both be losing coalitions in a weighted game where all winning coalitions of $v$ are winning coalitions, so that $\dim(v)\ge t$.         
\end{proof}

\begin{corollary}
  Let $\overline{n}=\left(n_1,\dots,n_t\right)\in\N_{>0}^t$ and $\overline{m}=\left(m_1,\dots,m_t\right)\in\N^t$ uniquely characterize a simple game $v$ with minimum. If $v$ does not contain
  veto players but contains null players, then we have $\dim(v)=t-1$.
\end{corollary}

\begin{theorem}
  Let $\overline{n}=\left(n_1,\dots,n_t\right)\in\N_{>0}^t$ and $\overline{m}=\left(m_1,\dots,m_t\right)\in\N^t$ uniquely characterize a simple game $v$ with minimum. If $v$ contains
  veto players but does not contain null players, then we have $\dim(v)=\max\{t-1,1\}$.
\end{theorem}
\begin{proof}
  Since $v$ does not contain null players, we have $1\le m_i\le n_i$ for all $1\le i\le t$. W.l.o.g.\ we assume $m_1=n_1$ and $m_i\le n_i-1$ for all $2\le i\le t$. Let $N_i$ denote the set of players corresponding to $n_i$. If $t=1$, then we have $v=[n_1;1,\dots,1]$, so that we assume $t\ge 2$ in the following. For each $2\le i\le t$ let $v^i$ denote the weighted game with quota $m_1+m_i$ where the players in $N_i\cup N_1$ have weight $1$ and all others have weight $0$. With this we have $v=\cap_{2\le i\le t} v^i$, so that $\dim(v)\le t-1$.
  
  For the other direction, we can assume $t\ge 3$. 
  For each index $2\le i\le t$ let $L_i\subseteq N$ such that $L_j=N_j$ for all $1\le j\le n$ with $j\neq i$ and $\left|L_i\cap N_i\right|=m_i-1$, so that $L_i$ is a losing coalition. For each 
  $1\le 2<j\le t$ let $a\in \left(L_i\backslash L_j\right) \cap N_j$ and $b\in \left(L_j\backslash L_i\right) \cap N_i$. With this, $L_i\backslash \{a\}\cup\{b\}$ and $L_j\backslash \{b\}\cup\{a\}$ 
  are winning coalitions. Thus, $L_i$ and $L_j$ cannot both be losing coalitions in a weighted game where all winning coalitions of $v$ are winning coalitions, so that $\dim(v)\ge t-1$. 
\end{proof}  

\begin{corollary}
  Let $\overline{n}=\left(n_1,\dots,n_t\right)\in\N_{>0}^t$ and $\overline{m}=\left(m_1,\dots,m_t\right)\in\N^t$ uniquely characterize a simple game $v$ with minimum. If $v$ contains
  both veto and null players, then we have $\dim(v)=\max\{t-2,1\}$.
\end{corollary}

\section{Conclusion}\label{S:conclusion}

In this article, we find formulas for the number of non-isomorphic simple games with minimum and its dimension. These goals are achieved using an algebraic parameterization of these simple games in terms of a vector and a matrix satisfying some properties. To obtain the enumeration, we used generating functions and Pólya Enumeration Theorem to remove isomorphic simple games.

These combinatorial techniques could be used to extend these enumerations to wider classes of simple games. For example, trying to consider the case where $r=2$ or directly $r\geq 2$. However, we need to point out that $r$ grows very fast (see  \cite{KroSu1995} and \cite{KuTau2013} for more details\footnote{In these papers, there is a bound on the subclass of complete simple games and, instead of minimal winning vector, in terms of a variation, called shift minimal winning vectors. The number of shift minimal winning coalitions, in a complete simple game is, at the same time, less or equal than the number of minimal winning coalitions of the same game, so we can use this bound to see how the number $r$ grows (at least).}).

Many real-world situations can be modeled by simple games with minimum, but sometimes, adding some other condition that, at first sight, can look simple, we can get out of this class. To illustrate why, to make this generalization is interesting, we continue with an example.

The bicameral context where the approval of two chambers is required, is a classical example of a bipartite simple game with one minimal winning vector, where the representatives of each chamber form a class of players, and a majority is required in both chambers (for more details see e.g. Section $7$ of \cite{FrSa21DAM}).

If we focus on the United States congress, see \cite{legislative_branch}. It says ``in order to pass legislation and send it to the President for his or her signature, both the House and the Senate must pass the same bill by majority vote". However, while the House is formed by an odd number of representatives, $435$, the Senate is formed by $100$ senators and in case of draw, an untie mechanism is defined consisting in the vote of the vice-president of the United States (who is not a senator), see U.S. Constitution, Article I, Section 3, Clause 4, e.g. \cite{vice_presindent_role}.
In that case we would have a pair $(\overline{n}_1,\mathcal{M}_1)$, where $\overline{n}_1=(435,101)$, where the first class of equivalent players corresponds to the House representatives and the second class corresponds to the senators and the Vice-President, and a matrix like 

         $$
\mathcal{M}_1=\begin{pmatrix}
 218 & 51 
\end{pmatrix}
$$

since, according to the desirability relation and the definition of the parameterization, the U.S. Vice-President is in the same class of the senators.

However, again according to \cite{legislative_branch}, ``If the President vetoers a bill, they [the House and the Senate] may override his veto by passing the bill again in each chamber with at least two-thirds of each body voting in favor". Hence, if we ask for the game consisting on passing the bill and not just sending it to the President for its signature, the game can be represented by $(\overline{n}_2,\mathcal{M}_2)$:

         $$
\mathcal{M}_2=\begin{pmatrix}
1 & 218 & 51 & 0\\
1 & 218 & 50 & 1\\
0 & 290 & 67 & 0
\end{pmatrix}
$$
with $\overline{n}_2=(1,435,100,1)$, where the classes of players correspond, in order, to the President, the House representatives, the Senate representatives and the Vice-president. Note that, in the last minimal winning vector, the Vice-President does not play any role since a tie is not possible when looking for two thirds of a non-multiple of three. Thus, we have mentioned a simple game with three minimal winning vectors, abandoning the class studied in this paper.


\end{document}